\documentclass[11pt]{amsart}

\usepackage{mathpazo}
\usepackage[mathpazo]{flexisym}
\usepackage{breqn}

\usepackage{verbatim}
\usepackage{graphicx}
\usepackage{xcolor}
\usepackage{amsfonts,amsmath,latexsym,amssymb,amsthm}
\usepackage{hyperref}

\newcounter{theoremcounter}

\newcounter{dummycounter}

\newcounter{emptycounter}

 \newtheorem{theorem}{Theorem}[section]
 \newtheorem{lemma}{Lemma}[section]
 \newtheorem{corollary}{Corollary}[section]
 
 \newtheorem{proposition}{Proposition}[section]

\newcounter{eqncounter}

\numberwithin{equation}{eqncounter}

\def\IR{\mathbb R}
\def\IC{\mathbb C}
\def\IZ{\mathbb Z}

\def\IQ{\mathbb Q}

\def\IU{\mathbb U}
\def\IS{\mathbb S}

\def\Pc{\mathcal{P}}

\def\IL{\mathcal{L}}

\def\H{\mathcal{H}}
\def\J{\mathcal{J}}
\def\I{\mathcal{I}}
\def\R{\mathcal{R}}
 \newcommand{\Norm}{N_{K/k}}
 \newcommand{\KT}{K^{\times}/\Tor(K^{\times})}

\def\p{\mathfrak{p}}
\def\q{\mathfrak{q}}
\def\B{\mathfrak{B}}
\def\A{\mathfrak{A}}
\def\C{\mathfrak{C}}
\def\D{\mathfrak{D}}
\def\E{\mathfrak{E}}
\def\F{\mathfrak{F}}

\def\Zstar{Z^*}
\def\Z{Z}
\def\Sa{\mathfrak{S}}
 \newcommand{\G}{\mathcal{G}}

\def\La{\Lambda}

\renewcommand{\vec}[1]{\mbox{\boldmath$#1$}}
\def\NKk{\mathcal{S}_{K/k}}

\def\id{\mathop{\rm id}\nolimits}

\def\Tor{\mathop{\rm Tor}\nolimits}

\def\Cl{\mathop{\rm Cl}\nolimits}
\def\Gal{\mathop{\rm Gal}\nolimits}
\def\IM{\mathop{\rm Im}\nolimits}
\def\ord{\mathop{\rm ord}\nolimits}

\def\Vol{\textup{Vol}}

\def\Oseen{{\mathcal{O}}}

\def\Sc{\mathcal{S}}

\def\bx{\vec{x}}

\def\bz{\vec{z}}
\def\vx{{\vec{x}}}
\def\vy{{\vec{y}}}

\def\bsigma{{\vec{\sigma}}}
\def\dis{D_{K/k}}
\def\conSK{A_K}

\def\Gal{\mathop{{\rm Gal}}\nolimits}

\title{Effective equidistribution of norm one elements in CM-fields} 
\author{Shabnam Akhtari, Jeffrey~D.~Vaaler, and Martin Widmer}
\address{Department of Mathematics, Pennsylvania State University, University Park, PA 16802 USA}
\email{akhtari@psu.edu}

\address{Department of Mathematics, University of Texas, Austin, TX 78712 USA}
\email{vaaler@math.utexas.edu}

\address{Graz University of Technology, Institute of Analysis and Number Theory, Steyrergasse 30/II, 8010 Graz, Austria}
\email{martin.widmer@tugraz.at}

\begin{document}

\date{\today}



\maketitle

\begin{abstract}
For a number field $K$ let $\Sc_K$ be the maximal  subgroup of the multiplicative group $K^\times$
that embeds into the unit circle under each embedding of $K$ into the complex numbers. 
The group $\Sc_K$ can be seen as an archimedean counterpart to the group of units  
$\Oseen_K^\times$ of the ring of integers $\Oseen_K$. If $K=\IQ(\Sc_K)$ is a CM-field then 
$\Sc_K/\Tor(K^\times)$ is a free abelian group of infinite rank.
If $K=\IQ(\Sc_K)$ is not a CM-field then $\Sc_K=\{\pm 1\}$. 
In the former case $\Sc_K$ is the kernel of the relative norm map from $K^\times$
to the multiplicative subgroup $k^\times$ of the maximal totally real subfield $k$ of $K$.

We prove an effective equidistribution result for the elements of $\Sc_K$ embedded into the
complex unit circle and enumerated by the Weil height. Our result also includes
a specific rate of convergence.

For imaginary quadratic fields an ineffective version of the equidistribution result  has been proven by Petersen and Sinclair.
\end{abstract}

\section{Introduction}
Let $K$ be a  number field. We choose a representative $|\cdot|_v$ for each place $v$ of $K$ and write $v|\infty$ if $v$ is archimedean and $v\nmid \infty$ if $v$ is non-archimedean. The group of units of the ring of integers $\Oseen_K$ is given by
$$\Oseen_K^\times=\{\alpha\in K^{\times}; |\alpha|_v=1 \; \text{ for all $v$ such that } v\nmid \infty \}.$$
In this article we study its ``archimedean counterpart'' defined by
\begin{alignat*}1
\Sc_K=\{\alpha\in K^{\times}; |\alpha|_v=1 \; \text{ for all $v$ such that } v|\infty \}.
\end{alignat*}
It is clear that $\Sc_K$ is also a subgroup of the multiplicative group $K^\times$ and that
\begin{alignat}1\label{eq:TorSK}
\{\pm1\}\subseteq \Tor(K^\times)=\Sc_K\cap \Oseen_K^\times,
\end{alignat}
where $\Tor(K^\times)$ denotes the torsion subgroup of $K^\times$.
Our main result is concerned with the distribution of elements in $\Sc_K$ when enumerated by the Weil height. 

But first let us clarify the basic structure of 
the group $\Sc_K$ and its connection to CM-fields. 
Recall that a number field $K$ is a CM-field if it is totally complex and contains a totally real
subfield $k$ of index $2$.
\begin{proposition}\label{prop:SKbasic}
Let $K$ be a number field, and let
$\IQ(\Sc_K)$ be the intersection of all subfields of $K$ that
contain $\Sc_K$.
 If $\IQ(\Sc_K)$ is not a CM-field then $\Sc_K = \{\pm 1\}$.
If $\IQ(\Sc_K)$ is  a CM-field then $\Sc_K/\Tor(K^\times)$ is a free abelian group of countably infinite rank.
\end{proposition}
Obviously we have $\Sc_{\IQ(\Sc_K)}=\Sc_K$ for each number field $K$.
For the purposes of studying the group $\Sc_K$ we can and will therefore assume that $K$ is a CM-field, and we write $k$ for its maximal totally real subfield. 

The norm map $\Norm:K^\times\to k^\times$ is a homomorphism of groups
and it is closely related to the group $\Sc_K$ via the following proposition.
The proof follows easily from a characterisation of CM-fields due to Shimura,
and stated here as Proposition \ref{prop: normkern2}.

\begin{proposition}\label{prop:SKnormkern}
Let $K$ be a CM-field, and let $k$ be its maximal totally real subfield. Then $\Sc_K$
is the kernel of the norm map $\Norm:K^\times\to k^\times$.
\end{proposition}

We now describe our main result. In short, it provides the asymptotics, and a power saving error term, for the number of elements in $\Sc_K$ with bounded height whose embeddings lie in given arcs
of the unit circle in $\IC$. 

Let $2N$ be the degree of the CM-field  $K$, and recall that $k$ is its maximal totally real subfield. 
For $1\leq n\leq N$ let $\sigma_n,\sigma_{n+N}$ be the $N$ pairs of complex conjugate embeddings of $K$ into $\IC$, so that 
$$\Sc_K=\{\alpha\in K; |\sigma_n(\alpha)|=1 \; \text{ for } 1\leq n\leq N\}.$$ 
 For a non-zero complex number $x$ we write $\arg(x)$ for the unique argument of $x$ in $[0,2\pi)$, so that $x=|x|e^{i\arg(x)}$.
For a product of intervals $\I=\I_1\times \cdots \times \I_N$ with each interval $\I_j\subseteq [0,2\pi)$ we define
\begin{alignat*}1
\Sc_K(\I)=\{\alpha\in \Sc_K; (\arg(\sigma_n(\alpha)))_n\in \I\},
\end{alignat*}
and we write $|\I|$ for the product of the lengths of the intervals $\I_1,\ldots, \I_N$.

For each place $v$ of $K$ we choose the unique representative $|\cdot |_v$ that either
extends the usual archimedean absolute value or one of the usual $p$-adic absolute values on $\IQ$,
and we write $[K_v:\IQ_v]$ for the local degree at $v$.
Let
\begin{alignat*}1
H(\alpha)=\prod_{v}\max\left\{1,|\alpha|_v\right\}^{\frac{[K_v:\IQ_v]}{[K:\IQ]}}
\end{alignat*}
denote the absolute (multiplicative) Weil height on $K$. We refer the reader to \cite[Section 1.5]{BG} for more details
on the Weil height.
For $\H\geq 1$ 
we define
\begin{alignat}1\label{set: SKIH}
\Sc_K(\I,\H)=\{\alpha\in \Sc_K(\I);  H(\alpha)\leq \H\}.
\end{alignat}
We set
\begin{alignat}1\label{def: SK}
\conSK=\left(\prod_{P|\dis}\frac{2N_{k/\IQ}(P)}{N_{k/\IQ}(P)+1}\right)\frac{1}{\sqrt{N_{k/\IQ}(\dis)}}\frac{h_kR_k}{\omega_k\zeta_k(2) |\Delta_k|},
\end{alignat}
where the product\footnote{As usual, the empty product is interpreted as $1$.} runs over all prime ideals $P$ of $\Oseen_k$ dividing the relative discriminant $\dis$  of $K/k$,  and
$N_{k/\IQ}(P)=[\Oseen_k:P]$ denotes the (absolute) norm of the ideal $P$,
$h_k$ denotes the class number, $R_k$ is the regulator,  $\omega_k=2$ is the number of roots of unity,  and $\Delta_k$ is the discriminant of $k$.

We are now in position to state the main result.
\begin{theorem}\label{thm: main}
Let $K$ be a CM-field of degree $2N$.
There exists  $C_K>0$, depending only on $K$, such that for every  $\H\geq 2$ we have
\begin{alignat}1\label{set: SKIH}
\left|\#\Sc_K(\I,\H)-\conSK|\I|\H^{2N}\right|\leq C_{K}\H^{2N-1}\IL,
\end{alignat}
where $\IL$ is defined to be $\log \H$ if $N=1$ and $1$ if $N\geq 2$.
\end{theorem}

For $\I=[0,2\pi)^N$ the main term of our result could possibly be derived from work of Batyrev and Tschinkel (\cite[Corollary 4.7]{BatyrevTschinkel1998} or even its precursor 
\cite{BatyrevTschinkel1995}). However, this would require some efforts. 
While very general, the methods from \cite{BatyrevTschinkel1998} do not provide effective results. To obtain an explicit power saving error term 
we develop another method, more in the spirit of \cite{MasserVaaler2}, which we explain at the end of this section.

Let $\Sc_K(\H)=\{\alpha\in \Sc_K; H(\alpha)\leq \H\}$.
Consider the discrepancy 
$$D_\H(\Sc_K)=\sup_{\I}\left|\frac{\#\Sc_K(\I,\H)}{\#\Sc_K(\H)}-\frac{|\I|}{(2\pi)^N}\right|\subseteq [0,1],$$
where the supremum is taken over all products of intervals 
$$\I=\I_1\times \cdots \times \I_N\subseteq [0,2\pi)^N.$$

Next consider the complete collection $\sigma_1,\ldots,\sigma_N$ of independent embeddings of $K$. 
Theorem \ref{thm: main} implies not only that the points of $\Sc_K$ are simultaneously and independently equidistributed on the unit circle under these embeddings (when enumerated by the Weil height)
but we also get a an explicit upper bound on the discrepancy (at least up to the constant $C'_K$).
\begin{corollary}\label{cor: discrep}
Let $K$ and $\IL$ be as in Theorem \ref{thm: main}. There exists $C'_K>0$ such that for $\H\geq 2$ we have
\begin{alignat}1\label{eq: discrp}
D_\H(\Sc_K)\leq C'_K \frac{\IL}{\H}.
\end{alignat}
\end{corollary}
A  precursor of Corollary \ref{cor: discrep} was proven
in 2011 by Petersen and Sinclair \cite[Theorem 2.1]{PetersenSinclair2011}
in the case of imaginary quadratic fields $K$. It is conceivable that  equidistribution can be deduced for arbitrary CM-fields by combining an observation of Peyre \cite[Proposition 5.0.1]{Peyre1995} with the aforementioned work of Batyrev and Tschinkel \cite{BatyrevTschinkel1998}.
However, neither Batyrev, Tschinkel and Peyre's nor Petersen and Sinclair's work yields an effective equidistribution result.
Petersen and Sinclair's work is analytic in nature and uses Weyl's equidistribution
criterion, the Wiener-Ikehara Tauberian Theorem, and properties of Hecke L-functions.
Our approach is quite different and the basic strategy is explained at the end of this section.

It is worthwhile to note that equidistribution 
fails if we consider many quadratic CM-fields simultaneously. 
Let  us consider the set  of all rational and all imaginary quadratic points on the unit circle  
$$\Sc_2=\bigcup_K \Sc_K,$$
where the union is taken over all imaginary quadratic fields $K$ (here we consider each 
$K$ as subfield of $\IC$),
and set
$$\Sc_2(\I,\H)=\{e^{i\theta}\in \Sc_2; \theta\in \I, H(e^{i\theta})\leq \H\}=\bigcup_K \Sc_K(\I,\H).$$
If $\I\subseteq (\pi,2\pi)$ then $\#\Sc_2(\I,\H)=\#\Sc_2(\I-\pi,\H)$ as $H(\alpha)=H(-\alpha)$.
Therefore, it suffices to consider the case $\I\subseteq [0,\pi]$.
 A point $e^{i\theta}\neq \pm 1$ on the unit circle is imaginary quadratic
if and only if $\cos(\theta)=-b/2a$, for coprime integers $a>0,b$. In this case the minimal polynomial is 
$$f(x)=ax^2-2a\cos(\theta)x+a=ax^2+bx+a\in \IZ[x]$$
and $H(e^{i\theta})=\sqrt{a}$ (see \cite[Propositions 1.6.5 and 1.6.6]{BG}). 
Writing $|\cos(\I)|$  for the length of the interval $\cos(\I)$, we get
$$\#\Sc_2(\I,\H)=O(1)+\sum_{a=1}^{\H^2}\sum_{b\in -2a\cos(\I) \atop (a,b)=1}1=\frac{|\cos(\I)|}{\zeta(2)}\H^4+O(\H^2\log \H).$$ 
In particular, $\Sc_2$ is not equidistributed  on the unit circle, when ordered by the Weil height $H(\cdot)$.\\

We conclude this section with a brief overview of the remaining sections.
In Section \ref{sec:basics} we recall some basic facts about CM-fields, and we deduce
the first part of Proposition \ref{prop:SKbasic} and  Proposition \ref{prop:SKnormkern}.

It follows from Proposition \ref{prop:SKnormkern} that Hilbert's Theorem 90 provides a 
surjective group homomorphism $\psi:K^\times \to \Sc_K$ with kernel $k^\times$.
In Section \ref{sec: grpstruc} we use this, in conjunction with the (logarithmic) Weil height, to deduce
that $\Sc_K/\Tor(K^\times)$ is a free abelian group of (countably) infinite rank, proving the
second part of Proposition \ref{prop:SKbasic}. 

Sections 4-7 are preparations
for the proof of Theorem \ref{thm: main}. Section \ref{sec: lattcount} provides the
counting principle Lemma \ref{lem: latticecounting} based on geometry of numbers to count lattice points.
In Section \ref{sec: prelim} we introduce the counting domain, and we prove that it satisfies the 
technical conditions needed to apply Lemma \ref{lem: latticecounting}. 

Section \ref{sec: funddom} can be seen as the core of the proof. The homomorphism $\psi$ induces an isomorphism $\hat{\psi}: K^\times/k^\times \to \Sc_K$. Therefore we need to construct a suitable fundamental domain of $K^\times$ under the
action of $k^\times$. ``Suitable'' means that the height bound cuts out a subset that is accessible 
to our counting techniques. All this is done in detail in Section \ref{sec: funddom} and (modulo minor 
modifications) this part is applicable to counting elements of bounded height in the kernel of the norm map for any quadratic extension $K/k$ of number fields. 

The next  step is to transform the counting problem to an ordinary lattice point counting problem, and this is carried out in Section \ref{sec: sieving}. We then have all in place to finalise the proof of Theorem \ref{thm: main}, which is done in Section \ref{sec: proofmainthm}.

In the final section we consider the  quotient group $K^\times/k^\times$.
We show that if a coset $k^\times$ in $K^\times$
intersects $\Sc_K$ then the
minimal height of all elements in that coset is the height of the elements that lie in $\Sc_K$
(clearly they all have equal height).
More generally we show that this holds true whenever
$K/k$ is a quadratic extension and $\Sc_K$ is the kernel of the norm map $\Norm: K^\times\to k^\times$.
Furthermore, we show that the cosets that intersect $\Sc_K$ are precisely the images of the 
squares in $\Sc_K$ under the inverse map of the isomorphism $\hat{\psi}$.

\section*{Acknowledgements}
Parts  of this work were done during a Summer Collaborators Program at the Institute for 
Advanced Study in  Princeton in 2024.
The authors gratefully acknowledge support from the IAS. 

M.W. would like  to thank Tim Browning, Christopher Frei, Daniel Loughran, Nick Rome, and Tim Santens for  helpful discussions and additional references.

\section{Basics on CM-fields}\label{sec:basics}
Let $\rho: \IC \to \IC$ be the complex conjugation. A basic observation made 
already by Shimura \cite[18.2. Lemma (i)]{shimura1998} is the following very useful characterisation of CM-fields.
\begin{lemma}[Shimura]\label{lem:Shimura}
A number field  $K$ is a CM-field if and only if there exists a non-trivial automorphism $\tau$ of $K$ such that $\sigma\circ\tau=\rho\circ \sigma$ for all homomorphisms $\sigma: K \to \IC$. 
\end{lemma}
If $K$ is a CM-field and $k$ its maximal totally real subfield then the automorphism
$\tau$ from Lemma \ref{lem:Shimura} satisfies
\begin{alignat}1\label{eq:tau}
\tau=\sigma^{-1}\circ \rho \circ \sigma
\end{alignat}
for every homomorphisms $\sigma: K \to \IC$, and it is  a non-trivial automorphism of 
$K$ fixing $k$ (we drop $\circ$ and simply write $\sigma^{-1}\rho\sigma$).
Consequently, $\tau$ is the unique non-trivial element of $\Gal(K/k)$, and
\begin{alignat}1\label{eq:gentau}
\Gal(K/k)=\langle \tau\rangle.
\end{alignat}
Shimura \cite[18.2. Lemma (ii)]{shimura1998} also observed that Lemma \ref{lem:Shimura} implies the following result. 
\begin{lemma}[Shimura]\label{cor:comp}
The composite field of finitely many CM-fields is also a CM-field. 
\end{lemma}

Blanksby and Loxton \cite[Theorem 1]{blanksby1978} proved a characterisation of CM-fields
that connects them to the group $\Sc_K$.
\begin{theorem}[Blanksby, Loxton]\label{thm:BlanksbyLoxton}
Let $K$ be a number field of degree $d>1$. Then 
$K$ is a CM-field if and only if $K=\IQ(\alpha)$ for some $\alpha\in \Sc_K.$
\end{theorem}
In fact \cite[Theorem 1]{blanksby1978} is stated slightly differently and we are using the fact that if the maximal modulus of the
conjugates (over $\IQ$) of an algebraic number $\alpha\in \IC$  is equal to $1$, then all conjugates lie on the unit circle. This is because
the complex conjugate $\rho(\alpha)$ is also a conjugate (over $\IQ$) of $\alpha$. Hence $\alpha$ is reciprocal.

Theorem \ref{thm:BlanksbyLoxton} in conjunction with Lemma \ref{cor:comp} yields the first part of Proposition \ref{prop:SKbasic}.
\begin{lemma}\label{lem:CMchar}
Let $K$ be a number field. Then $\Sc_K \neq \{\pm 1\}$
if and only if  $\IQ(\Sc_K)$  is a CM-field.   
\end{lemma}
\begin{proof}
Suppose $\alpha_1,\ldots,\alpha_n\in \Sc_K\backslash\{\pm 1\}$ with $\IQ(\Sc_K)=\IQ(\alpha_1,\ldots,\alpha_n)$. By Theorem \ref{thm:BlanksbyLoxton} we see that $K_i=\IQ(\alpha_i)$
is CM for $1\leq i\leq n$. By Lemma \ref{cor:comp} we conclude that $\IQ(\Sc_K)=K_1\cdots K_n$ is also CM. The other direction is trivial.
\end{proof}
 
 Next we restate the Proposition \ref{prop:SKnormkern} and we prove it.
\begin{proposition}\label{prop: normkern2}
Let $K$ be a CM-field and $k$ its maximal totally real subfield. Then $\Sc_K$ is the kernel of the norm map $\Norm:K^\times\to k^\times$.
\end{proposition}
\begin{proof}
Let $\alpha\in K^\times$ be in the kernel of the norm map. Using (\ref{eq:gentau}) and (\ref{eq:tau}) gives
$$1=\Norm(\alpha)=\alpha\tau(\alpha)=\alpha\sigma^{-1}(\rho(\sigma(\alpha)))$$
for all homomorphisms $\sigma: K \to \IC$.
Applying $\sigma$ on both sides gives 
$$1=\sigma(\alpha)\rho(\sigma(\alpha))=|\sigma(\alpha)|^2.$$ 
Hence, $|\alpha|_v=1$ for all archimedean places $v$ of $K$,
and so $\alpha\in \Sc_K$. 

Now suppose $\beta\in \Sc_K$. Then $|\sigma(\beta)|^2=\sigma(\beta)\rho(\sigma(\beta))=1$ for all homomorphism $\sigma: K \to \IC$.
Applying $\sigma^{-1}$ on both sides gives   $1=\beta\sigma^{-1}(\rho(\sigma(\beta)))=\beta\tau(\beta)$. Thus $\Norm(\beta)=1$.
\end{proof}
We also learn from this proof that if $\alpha\in K^\times$ then $\sigma(\Norm(\alpha))=|\sigma(\alpha)|^2>0$
for any  homomorphism $\sigma: K \to \IC$. Hence, the norm $\Norm$ maps to the subset of $k^\times$
of totally positive elements.

As before let $\tau$ be the unique non-trivial automorphism of $K$ fixing $k$, and let
$\psi:K^\times\to K^\times$ be the group homomorphism defined by  
\begin{alignat}1\label{map: psi}
\psi(\beta)=\frac{\beta}{\tau(\beta)}.
\end{alignat}
We note that the kernel of $\psi$ is  $k^\times$. Since $K/k$ is a cyclic extension and $\Gal(K/k)=\langle \tau\rangle$
it follows from Hilbert's Theorem 90 that $\ker \Norm=\IM \psi$. 
Hence, the maps $\psi$ and the norm $N_{K/k}$ yield an exact sequence
\begin{alignat}1\label{exactseq}
k^\times \overset{\id}\longrightarrow K^\times \overset{\psi}{\longrightarrow}\ K^\times \overset{N_{K/k}}{\longrightarrow}\ k^\times,
\end{alignat}
and we know from Proposition \ref{prop: normkern2} that the group $\Sc_K$  is given by  the kernel of the norm map $N_{K/k}$ which in turn is equal to the image of $\psi$. Hence,
\begin{alignat}1\label{eq:SKpsi}
\Sc_K=\IM \psi \cong K^\times/k^\times.
\end{alignat}

\section{The group structure of $\Sc_K$}\label{sec: grpstruc}
In this section we prove that $\Sc_K/\Tor(K^\times)$ is a free abelian group of countably infinite rank,
proving the second claim of Proposition \ref{prop:SKbasic}.

\begin{lemma}\label{lem:infrank}  Let $K$ be a CM-field and let $k$ be its maximal totally real subfield. 
Then $\Sc_K/\Tor(K^{\times})$ is a free abelian group of countably infinite rank.  
\end{lemma}

\begin{proof}  
It follows from the exact sequence (\ref{exactseq}) that the induced homomorphism
\begin{alignat}1\label{map:indhom}
\widehat{\psi} : K^{\times}/k^{\times} \rightarrow \Sc_K
\end{alignat}
is an isomorphism of multiplicative groups.  By a result of Brandis \cite{brandis1965} the group $K^{\times}/k^{\times}$ is not finitely generated.  
Because (\ref{map:indhom}) is an isomorphism, we conclude that the group $\Sc_K$ is  not finitely generated.  

From (\ref{eq:TorSK}) we get
\begin{alignat}1\label{eq:Tors}
\Tor(\Sc_K) = \Tor\bigl(K^{\times}\bigr),
\end{alignat}
and therefore $\Tor(\Sc_K)$ is a finite cyclic group of order $2q$ where $q | \Delta_K$.  
Hence, the torsion-free abelian group $\Sc_K/\Tor(K^\times)$ is not finitely generated.

We note that the absolute, logarithmic Weil height $h(\cdot)=\log H(\cdot)$ is
well defined on the multiplicative quotient group
\begin{equation*}\label{eq:quotgroup}
\KT = \G_K.
\end{equation*}
Moreover, the Weil height
\begin{alignat*}1
h : \G_K \rightarrow [0, \infty)
\end{alignat*}
satisfies (here we write $\alpha$ and $\beta$ for coset representatives in $\G_K$):
\begin{itemize}
\item[ ]
\item[(i)] $0 \leq h(\alpha)$ for $\alpha$ in $\G_K$, and $0 = h(\alpha)$ if and only if $\alpha = 1$,
\item[ ]
\item[(ii)] $h\bigl(\alpha^m\bigr) = |m| h(\alpha)$ for each $m \in \IZ$ and $\alpha$ in $\G_K$,
\item[ ]
\item[(iii)] $h\bigl(\alpha \beta\bigr) \leq h(\alpha) + h(\beta)$ for each $\alpha$ and $\beta$ in $\G_K$,
\item[ ]
\item[(iv)] there exists $0 < \varepsilon(K)$ so that $\varepsilon(K) \leq h(\alpha)$ for each $\alpha \neq 1$ in $\G_K$.
\item[ ]
\end{itemize}
These four conditions imply that $h$ is a  discrete norm on the abelian group $\G_K$, and on all of its 
subgroups. 
It is known (see \cite{lawrence1984}, \cite{steprans1985},
and \cite{zorzitto1985}) that an abelian group with a discrete norm must be a free group.   As
\begin{alignat*}1
\Sc_K/\Tor(K^\times) \subseteq \G_K,
\end{alignat*}
we find that the quotient group $\Sc_K/\Tor(K^\times)$ is a free group and not finitely generated. Hence, 
$\Sc_K/\Tor(K^\times)$ has (countably) infinite rank.  
\end{proof}

\section{Lattice point counting}\label{sec: lattcount}
Throughout this section let $D\geq 2$ be an integer.
By a lattice $\La$ in $\IR^D$ we mean a discrete, free $\IZ$-module of rank $D$.
Let $\lambda_1(\Lambda)$ be the shortest euclidean length of a non-zero vector of $\Lambda$ 
\begin{alignat*}3
\lambda_1=\min \{|\bx|;  \bx\in \Lambda, \bx\neq 0\}.
\end{alignat*}
Let $M$ be a positive integer,
and let $L$ be a non-negative real number.
We say that a set $S$ is in Lip$(D,M,L)$ if 
$S$ is a subset of $\IR^D$, and 
if there are $M$ maps 
$$\phi_1,\ldots,\phi_M:[0,1]^{D-1}\longrightarrow \IR^D$$
satisfying a Lipschitz condition
\begin{alignat*}3
|\phi_i(\vx)-\phi_i(\vy)|\leq L|\vx-\vy| \text{ for } \vx,\vy \in [0,1]^{D-1}, i=1,\ldots,M 
\end{alignat*}
such that $S$ is covered by the images
of the maps $\phi_i$. 

If the boundary $\partial S$ is in Lip$(D,M,L)$ then $\partial S$ has measure zero and thus $S$ is measurable (see, e.g., \cite{Spain}).

The following Lemma is  \cite[Lemma 3.1]{Widmer_Mathemtika_2018}.
\begin{lemma}\label{lem:lattcount}
Let $\Lambda$ be a lattice in $\IR^D$
Let $S$ be a set in $\IR^D$ such that
the boundary $\partial S$ of $S$ is in Lip$(D,M,L)$, and suppose $S$ lies in the closed euclidean 
ball with centre $P$.
Then $S$ is measurable, and moreover,
\begin{alignat*}3
\left|\#(\Lambda \cap S)-\frac{\Vol S}{\det \Lambda}\right|
\leq D^{3D^2/2} M\left(\left(\frac{L}{\lambda_1}\right)^{D-1}+1^*(S\cap \Lambda)\right),
\end{alignat*}
where $1^*(S\cap \Lambda)=0$ if $S\cap \Lambda=\emptyset$ and  $1^*(S\cap \Lambda)=1$
otherwise.
\end{lemma}
If we can choose $P$ to be the origin and the latter is not contained in $S$
then we can get rid of the extra $1^*$ in the error term, and this gives a slightly
more convenient version for our purposes.
\begin{lemma}\label{lem: latticecounting}
Let $\Lambda$ be a lattice in $\IR^D$ and  $\lambda_1=\lambda_1(\Lambda)$.
Let $S$ be a set in $\IR^D$ such that
the boundary $\partial S$ of $S$ is in Lip$(D,M,L)$. Suppose $S$ is contained in the closed euclidean ball about the origin of radius $L$,
and the origin is not contained in $S$.
Then
\begin{alignat*}3
\left|\#(\Lambda\cap S)-\frac{\Vol S}{\det \Lambda}\right|
\leq 2D^{3D^2/2} M\left(\frac{L}{\lambda_1}\right)^{D-1}.
\end{alignat*}
\end{lemma}
\begin{proof}
The claim follows from Lemma \ref{lem:lattcount} upon noticing that
if $L<\lambda_1$ then $\Lambda\cap S=\emptyset$ so that $1^*(\Lambda\cap S)=0$. 
And if $L\geq \lambda_1$ then $L/\lambda_1\geq 1^*(\Lambda\cap S)$.
\end{proof}

\section{Preliminaries}\label{sec: prelim}
For $1\leq n\leq N$ let $\sigma_n,\sigma_{n+N}$ be the $N$ pairs of complex conjugate embeddings of $K$ into $\IC$.
Write $\bsigma:K\to \IC^N$ for the Minkowski-embedding defined by 
$$\bsigma(\beta)=(\sigma_1(\beta),\ldots,\sigma_N(\beta)).$$
Recall that $k$,  the maximal totally real subfield of $K$, has degree $N$,
and its $N$ distinct embeddings into $\IC$ are given by the restrictions of  $\sigma_1,\ldots, \sigma_N$ to $k$. 
Let $\l:k^\times\to \IR^N$ be the usual logarithmic mapping defined by 
$$\l(\beta)=(2\log|\sigma_1(\beta)|,\ldots, 2\log|\sigma_N(\beta)|).$$ 
Let $F$ be a fundamental domain of $\Sigma=\{(z_n)_n\in \IR^N; \sum_n z_n=0\}$ for the action of
the subgroup $\IU_k=\l(\Oseen_k^\times)$ on $\Sigma$.
If $N=1$ (i.e., $k=\IQ$) we have $\IU_k=\{0\}$, and so there is no choice except $F=\Sigma=\{0\}$.
If $N>1$ then $\IU_k$ is a lattice in $\Sigma$ and we have many choices  for $F$. It is convenient to have an $F$ with ``simple'' geometry, therefore we take  
$$F=[0,1)u_1+\cdots +[0,1)u_{N-1}$$ 
where $(u_1,\ldots, u_{N-1})$ is a reduced basis of the unit lattice $\IU_k$, in the sense that $|u_1|\leq \cdots \leq |u_{N-1}|\leq c_N R_k$ for
some constant $c_N>0$. The existence of such a reduced basis follows from the general reduction theory and the additional fact $|u_1|\gg_N 1$ which is a consequence
of Northcott's Theorem.
Let $T\geq 1$. Let $([K_v:\IQ_v])_{v\mid \infty}=(2,\ldots,2)$, and
consider the vector sum
$$F(T)=F+(2,\ldots,2)(-\infty,\log T].$$ 
Then
$F(\infty)=F+(2,\ldots,2)(-\infty,\infty)$ is a fundamental domain for the action of the subgroup $\IU_K$ on $\IR^N$.

Let $\J=\J_1\times \cdots \times \J_N$ where each $\J_j$ is an arbitrary subset of $[0,2\pi)$.
We define the set
\begin{alignat}1\label{set: SF}
S_F(\J;T)=\{\bx=(x_n)_n\in (\IC^\times)^N; (2\log|x_n|)_n\in F(T), \text{ and }(\arg(x_n))_n\in \J\}.
\end{alignat}
We note that $S_F(\J;T)$ is homogeneously expanding, i.e.,
\begin{alignat}1\label{eq: SFhom}
S_F(\J;T)=TS_F(\J;1).
\end{alignat}

Now let $\I=\I_1\times \cdots \times \I_N\subseteq [0,2\pi)^N$ be a product of intervals as in Theorem \ref{thm: main}, and let
\begin{alignat}1\label{def: Istar}
\I^*=\frac{1}{2}\I_1\times \prod_{n=2}^N \left(\frac{1}{2}\I_n\cup \left(\frac{1}{2}\I_n+\pi\right)\right)
\end{alignat}
One step in the proof of Theorem \ref{thm: main} is to count lattice points inside the set $S_F(\I^*;T)$ for suitable $T$. To this end we will need the following two lemmas.
Recall from Section \ref{sec: lattcount}
that a set $S\subseteq \IR^D$ is in Lip$(D,M,L)$ if  
if there are $M$ maps 
$$\phi_1,\ldots,\phi_M:[0,1]^{D-1}\longrightarrow \IR^D$$
satisfying a Lipschitz condition
\begin{alignat*}3
|\phi_i(\vx)-\phi_i(\vy)|\leq L|\vx-\vy| \text{ for } \vx,\vy \in [0,1]^{D-1}, i=1,\ldots,M 
\end{alignat*}
such that $S$ is covered by the images
of the maps $\phi_i$.

\begin{lemma}\label{lem: Lip}
The set  $S_F(\I^*;1)$  is contained in the closed euclidean ball about the origin of radius $L$, and its boundary $\partial(S_F(\I^*;1))$ is in Lip$(2N,M,L)$ with $M=M(N)$ and $L=L(K)$ depending only on $K$. Further, the origin is not contained in $S_F(\I^*;1)$.
\end{lemma}
\begin{proof}
The last assertion is clear from the definition (\ref{set: SF}). 
The first and the second assertion are easy to see for $N=1$, so we assume $N\geq 2$. 
Thus $F=[0,1)u_1+\cdots +[0,1)u_{N-1}$ and $|u_i|\leq c_N R_k$ for $1\leq i \leq N-1$.

For the first assertion we note that if $\bx\in S_F(\I^*;1)$ then $|x_n|^2=\exp(z_n+2t)$ for some $\bz\in F$ and $t\in (-\infty,0]$.
Hence,  $|z_n|\leq (N-1)c_N R_k\leq Nc_NR_k$, and so the first claim holds for any 
$L\geq L_0:=(N\exp(Nc_NR_k))^{1/2}$.

Now let us prove the second claim. 
The boundary $\partial(S_F(\I^*;1))$ comes in two flavours. Firstly, those points 
$\bx$ in the topological closure of $S_F(\I^*;1)$
with $(|x_n|^2)_n\in \exp(\partial F(1))$, where we used $\exp$ for the diagonal exponential map from $\IR^{N}$ to $(0,\infty)^N$. And secondly, those points $\bx$ in the closure of $S_F(\I^*;1)$
with $\bx \in [0,1]\cdot L_0\exp(i\partial \I^*)$, where $\exp$ denotes the complex  diagonal exponential map,
and $\partial \I^*$ denotes the boundary of the set $\I^*\subseteq \IR^N$.

The latter are covered by $2^N$ Lipschitz maps as follows. Choose $1\leq m\leq N$, and
let $\gamma$ be one of the two endpoints of $\I_m$. Sending $t_m\in [0,1]$ to $t_m L_0\exp(i\gamma)$ and, for
$n\neq m$, sending $(t_n,\gamma_n)\in [0,1]^2$ to $t_nL_0\exp(i2\pi \gamma_n)$ defines a map from $[0,1]^{2N-1}$ to $\IC^N$.  In this way we get $2^N$ maps whose images cover 
$[0,1]\cdot L_0\exp(i\partial \I^*)$,
and each one satisfies Lipschitz condition with Lipschitz constant $L\ll_K 1$ (see \cite[(1)-(3) Appendix A]{art1}).
 
Parametrising the  points of the first kind is more involved but this has been done (in a more general setting)
in \cite[Lemma 3]{MasserVaaler2} and, with explicit constants, in \cite[Lemma A.1]{art1} (with the irrelevant difference that  the coordinates $x_n$ are in some $\IC^{m}$ instead of $\IC$).
\end{proof}

\begin{lemma}\label{lem: volume}
The set $S_F(\I^*;1)$ is measurable and we have
\begin{alignat*}1
\Vol(S_F(\I^*;1))=\frac{|\I| R_k}{2^{N}\omega_k}.
\end{alignat*}
\end{lemma}
\begin{proof}
We note that the $(N-1)$-volume of $F$ is $\sqrt{N}R_k$.
The proof is now nearly identical to the proof of \cite[Lemma 4]{MasserVaaler2} and left to the reader.
\end{proof}

\section{Constructing a suitable fundamental domain}\label{sec: funddom}

We use the letters $A,B,C$ to denote non-zero ideals in $\Oseen_k$,  and $P,Q$ to denote  non-zero prime ideals in $\Oseen_k$.
And we use the letters $\A,\B,\C,\D,\Pc$  to denote non-zero  ideals in $\Oseen_K$, and the letters $\p,\q$ to denote non-zero  prime ideals in $\Oseen_K$.
We write $A\Oseen_K$ for the extension of the ideal $A\subseteq \Oseen_k$ to an ideal of $\Oseen_K$, and we note that $(AB)\Oseen_K=(A\Oseen_K)(B\Oseen_K)$.
Let 
\begin{alignat*}1
\R_k=\{C_1,\ldots,C_h\} 
\end{alignat*}
be a complete system of integral inequivalent representatives of the class group $\Cl_k$ of $k$.
Let 
\begin{alignat*}1
P_1, \ldots, P_s
\end{alignat*}
be the (possibly empty) list of prime ideals of $\Oseen_k$ that ramify in $K$, so that
\begin{alignat*}1
P_1\Oseen_K=\p_1^2, \ldots, P_s\Oseen_K=\p_s^2
\end{alignat*}
for certain distinct prime ideals $\p_1,\ldots,\p_s$ of $\Oseen_K$. 
We note that $P_1,\ldots, P_s$ are precisely the prime ideals that divide $\dis$.
We set
\begin{alignat*}1
\Pc=\p_1\cdots \p_s,
\end{alignat*}
for the square-free part of $\dis$,
and the empty product (i.e., $s=0$) is understood as $\Oseen_K$.

We let $I_\Pc$ be the set of  non-zero ideals of $\Oseen_K$ that have no ideal divisors $A$ defined over the subfield $k$ and are coprime to $\Pc$, i.e., 
\begin{alignat*}1
I_\Pc=\{\B\subseteq \Oseen_K; \B\neq \{0\}, (\B,\Pc)=1, \text{ and } A\Oseen_K\nmid \B \text{ for all }A\subsetneq \Oseen_k\}.
\end{alignat*}

\begin{lemma}\label{lem: uniquedecomposition}
Let $A,A'$ be non-zero ideals of $\Oseen_k$, let $\D,\D'$ be ideals of $\Oseen_K$ both dividing $\Pc$, and let $\B$ and $\B'$ both be in $I_\Pc$.
If $A\Oseen_K\D\B=A'\Oseen_K\D'\B'$ then $A=A'$, $\D=\D'$ and $\B=\B'$.
\end{lemma}
\begin{proof}
If  $\p|\Pc$ then $\ord_\p(A\Oseen_K\B)$ and $\ord_\p(A'\Oseen_K\B')$ are even.
This implies that $\D=\D'$, and thus 
$A\Oseen_K\B=A'\Oseen_K\B'$. Let $P\subseteq \Oseen_k$, and suppose $\ord_P(A)\geq \ord_P(A')$.
Dividing both sides by $P^{\ord_P(A')}\Oseen_K$, and assuming $\ord_P(A)> \ord_P(A')$ we conclude
$P\Oseen_K|\B'$ which is impossible as $\B'\in I_\Pc$.
This proves that $A=A'$, and hence $\B=\B'$.
\end{proof}

Recall that $\tau$ is the unique non-trivial automorphism of $K$ fixing $k$, and  
\begin{alignat}1\label{map: psi}
\psi:K^\times\to K^\times
\end{alignat}
denotes the group homomorphism
defined by $\psi(\beta)=\beta/\tau(\beta)$.
The kernel of $\psi$ is  $k^\times$ and $\tau^2=\id$.

\begin{lemma}\label{lem: BtauB}
Let $C\in \R_k$, $\D|\Pc$ and $\B\in I_\Pc$. Then $\tau(C\Oseen_K\D\B)=C\Oseen_K\D\tau(\B)$, and $(\B,\tau(\B))=1$.
\end{lemma}
\begin{proof}
We have
$\tau(C\Oseen_K)=\tau(C)\Oseen_K=C\Oseen_K$. Further, $\tau(\p_j)^2
=\tau(\p_j^2)=\tau(P_j\Oseen_K)=\tau(P_j)\Oseen_K=P_j\Oseen_K=\p_j^2$, 
and thus $\tau(\p_j)=\p_j$, and so $\tau(\D)=\D$. 
Now suppose the prime ideal $\q$ divides $\B$ and $\tau(\B)$. Then $\tau(\q)$ also divides $\tau^2(\B)=\B$. As $\B\in I_\Pc$ it follows that
$\q\cap k$ must split in $K$, but then $\q$ and $\tau(\q)$ are distinct prime ideals and thus
$Q\Oseen_K=\q\tau(\q)$ divides $\B$, contradicting  that $\B\in I_\Pc$. Hence $\B$ and $\tau(\B)$ are coprime.
\end{proof}

\begin{lemma}\label{lem: psibasics}
If $\beta\in K^\times$ and $1\leq n\leq N$, then
\begin{alignat*}1
\sigma_n(\psi(\beta))=\frac{\sigma_n(\beta)}{\rho(\sigma_n(\beta))}.
\end{alignat*}
In particular, $\arg(\sigma_n(\psi(\beta)))$ and $2\arg(\sigma_n(\beta))$ differ by an integer multiple of $2\pi$.
\end{lemma}
\begin{proof}
Using (\ref{eq:tau}) with $\sigma=\sigma_n$ gives
\begin{alignat*}1
\sigma_n(\psi(\beta))=\frac{\sigma_n(\beta)}{\sigma_n(\tau(\beta))}=\frac{\sigma_n(\beta)}{\rho(\sigma_n(\beta))}.
\end{alignat*}
\end{proof}

 Next we define 
\begin{alignat*}1
Z^*(\A)=\{\beta\in \A\backslash\{0\}; \beta\Oseen_K=\A\B \text{ for some } \B\in I_\Pc\}.
\end{alignat*}

\begin{lemma}\label{lem: psibetaheight}
Let $C\in \R_k$ and $\D|\Pc$.
Let $\beta$ be in $\Zstar(C\Oseen_K\D))$. Then
$$H(\psi(\beta))=\left(\frac{N_{K/\IQ}(\beta)}{N_{K/\IQ}(C\Oseen_K\D)}\right)^{1/(2N)}.$$
\end{lemma}
\begin{proof}
By Lemma \ref{lem: psibasics} 
$\psi(\beta)\in \Sc_K$, and so there are no archimedean contributions to the height of $\psi(\beta)$.
Writing $v$ for the places of $K$ and $d_v=[K_v:\IQ_v]$, we have
$$H(\psi(\beta))=\prod_{v\nmid \infty}\max\left\{1,\left|\frac{\beta}{\tau(\beta)}\right|_v\right\}^{d_v/(2N)}.$$ 
By Lemma \ref{lem: BtauB} we have 
$$\frac{\beta\Oseen_K}{\tau(\beta\Oseen_K)}=\frac{\B}{\tau(\B)},$$
and  $\B$ and $\tau(\B)$ are coprime.
It follows that
$$\prod_{v\nmid \infty}\max\left\{1,\left|\frac{\beta}{\tau(\beta)}\right|_v\right\}^{d_v/(2N)}=\left(N_{K/\IQ}(\tau(\B))\right)^{1/(2N)}.$$ 
Now $N_{K/\IQ}(\tau(\B)=N_{K/\IQ}(\B)$, and thus
$$H(\psi(\beta))=\left(N_{K/\IQ}(\B)\right)^{1/(2N)}=\left(\frac{N_{K/\IQ}(\beta)}{N_{K/\IQ}(C\Oseen_K\D)}\right)^{1/(2N)}.$$
\end{proof}

As $F, \I$ and $\H$ are kept fixed we may 
simplify the notation and write 
\begin{alignat}1\label{def: SCD}
\IS_{C,\D}=\{\beta\in C\Oseen_K\D; \bsigma(\beta)\in S_F(\I^*;\H N_{K/\IQ}(C\Oseen_K\D))^{1/(2N)}\},
\end{alignat}
where $\I^*$ was defined in (\ref{def: Istar}).

\begin{lemma}\label{lem: psirestriction}
The restriction of the map $\psi$ defined in  (\ref{map: psi}) to the subset 
$$\bigcup_{C\in \R_k}\bigcup_{\D|\Pc}(\Zstar(C\Oseen_K\D))\cap \IS_{C,\D}$$
maps to the set $\Sc_K(\I,\H)$.
\end{lemma}
\begin{proof}
Let $C\in \R_k$ and $\D|\Pc$, and let $\beta$ be in $\Zstar(C\Oseen_K\D))\cap \IS_{C,\D}$.
Recalling (\ref{def: SCD}) and (\ref{def: Istar}) we see that for $1\leq n\leq N$
$$\arg(\sigma_n(\beta))\in \I^*_n \subseteq \frac{1}{2}\I_n\cup \left(\frac{1}{2}\I_n+\pi\right),$$
so that $2\arg(\sigma_n(\beta))\in \I_n\cup (\I_n+2\pi)$. It follows from
Lemma \ref{lem: psibasics}
that  $\arg(\sigma_n(\psi(\beta)))\in \I_n$ for $1\leq n\leq N$.
And  it follows from Lemma \ref{lem: psibetaheight} that
$$H(\psi(\beta))=\left(\frac{N_{K/\IQ}(\beta)}{N_{K/\IQ}(C\Oseen_K\D)}\right)^{1/(2N)}\leq \H.$$
This proves that $\psi(\beta)\in \Sc_K(\I;\H)$.
\end{proof}

Lemma \ref{lem: psirestriction}
shows that we have the map
\begin{alignat}1\label{map: psirestriction}
\psi:\bigcup_{C\in \R_k}\bigcup_{\D|\Pc}(\Zstar(C\Oseen_K\D))\cap \IS_{C,\D}\to \Sc_K(\I,\H)
\end{alignat}

\begin{lemma}\label{lem: surjective}
The map $\psi$ defined in (\ref{map: psirestriction}) is surjective.
\end{lemma}
\begin{proof}
Let $\alpha$ be in $\Sc_K(\I,\H)$.
Since $\Sc_K$ is the kernel of the norm map $N_{K/k}$ it follows from the exact sequence
(\ref{exactseq}) that
$\psi:K^\times\to K^\times$ has image $\Sc_K$ and kernel $k^\times$. 
Hence there is $\beta\in \Oseen_K$ that maps to $\alpha$.
Now let $A\subseteq \Oseen_k$ be of maximal norm with $A\Oseen_K|\beta\Oseen_K$. Hence, $\p^2\nmid \beta\Oseen_K(A\Oseen_K)^{-1}$ whenever $\p|\Pc$.
Therefore there exists $\D|\Pc$ and $\B\in I_\Pc$ such that 
$$\beta\Oseen_K=A\Oseen_K\D\B.$$
There exists $\xi\in k^\times$, unique up to units of $\Oseen_k$, such that $\xi A\in \R_k$, say $\xi A=C$.
Replacing $\beta$ by $\xi\beta$ we get 
$$\beta\Oseen_K=\xi\Oseen_K A\Oseen_K\D\B=(\xi A)\Oseen_K\D\B=C\Oseen_K\D\B,$$ 
and this shows that
$\beta \in \Zstar(C\Oseen_K\D)$.
Multiplying $\beta$ with a unit of $\Oseen_k$, unique up to sign, we get $(2\log|\sigma_n(\beta)|)_n\in F(\infty)$ and, of course, still $\beta \in \Zstar(C\Oseen_K\D)$.
Replacing $\beta$ by $-\beta$ if needed we get $\bsigma(\beta)\in S_F(\I^*; \infty)$.  
Finally, by Lemma \ref{lem: psibetaheight} 
$$H(\psi(\beta))=\left(\frac{N_{K/\IQ}(\beta)}{N_{K/\IQ}(C\Oseen_K\D)}\right)^{1/(2N)}.$$ Since $H(\psi(\beta))=H(\alpha)\leq \H$,
we conclude that  $\beta\in \IS_{C,\D}$. This proves the surjectivity of (\ref{map: psirestriction}). 
 \end{proof}

\begin{lemma}\label{lem: injective}
The map $\psi$ defined in (\ref{map: psirestriction}) is injective.
\end{lemma}
\begin{proof}
Suppose $\beta,\beta'$ are both in the domain and have the same image. Then $\beta'/\beta\in k^\times$ and so $\beta'=\xi\beta$ for some $\xi \in k^\times$.
Further 
\begin{alignat*}1
\beta\Oseen_K=C\Oseen_K \D\B, \quad \text{ and } \quad 
\beta'\Oseen_K=C'\Oseen_K \D'\B'
\end{alignat*}
for certain $C,C'$ both in $\R_k$, $\D,\D'$ both dividing $\Pc$, and $\B,\B'$ both in $I_\Pc$.
Hence, 
$$\xi C\Oseen_K \D\B=C'\Oseen_K \D'\B'.$$
Writing $\xi=\xi_1/\xi_0$ with non-zero  $\xi_0,\xi_1\in \Oseen_k$, and $A=\xi_1 C$, $A'=\xi_0 C'$ it follows that
\begin{alignat*}1
A\Oseen_K\D\B=A'\Oseen_K\D'\B'.
\end{alignat*}
Now
Lemma \ref{lem: uniquedecomposition} tells us that
that $A=A'$, $\D=\D'$ and $\B=\B'$. This implies that $C$ and $C'$ both lie in the same ideal class of $k$ and so must be equal.
Consequently $\beta'=\eta \beta$ for a unit $\eta$ in $\Oseen_K^\times\cap k=\Oseen_k^\times$.  As $\beta$ and $\beta'$ are both in $ \IS_{C,\D}$ we 
conclude that $\l(\eta)=0$ and thus $\eta=\pm 1$. Finally, since $\arg(\sigma_1(\beta))$ and $\arg(\sigma_1(\beta'))$ are both in $(1/2)\I_1\subseteq [0,\pi)$ 
the case $\beta'=-\beta$ is ruled out. This proves the injectivity. 
\end{proof}

\begin{lemma}\label{lem: disjoint}
The union
$$\bigcup_{C\in \R_k}\bigcup_{\D|\Pc}(\Zstar(C\Oseen_K\D))$$
is disjoint.
In particular, 
$$\bigcup_{C\in \R_k}\bigcup_{\D|\Pc}(\Zstar(C\Oseen_K\D))\cap \IS_{C,\D}$$
is a disjoint union.
\end{lemma}
\begin{proof}
Let $C_1,C_2\in \R_k$, let $\D_1,\D_2$ both be divisors of $\Pc$ and suppose 
$$\beta\in \Zstar(C_1\Oseen_K\D_1)\cap \Zstar(C_2\Oseen_K\D_2).$$ 
Hence, 
there exist $\B_1,\B_2\in I_\Pc$ such that
$C_1\Oseen_K\D_1\B_1=C_2\Oseen_K\D_2\B_2$. 
Now Lemma \ref{lem: uniquedecomposition} implies that 
$C_1=C_2$ and $\D_1=\D_2$ and this proves the lemma.
\end{proof}

\begin{lemma}\label{lem: SKcount}
We have 
\begin{alignat*}1
\#\Sc_K(\I,\H)=\sum_{C\in \R_k}\sum_{\D|\Pc}\#\Zstar(C\Oseen_K\D)\cap \IS_{C,\D}.
\end{alignat*}
\end{lemma}
\begin{proof}
Follows immediately from Lemma \ref{lem: surjective}, \ref{lem: injective}, and \ref{lem: disjoint}.
\end{proof}

\section{Sieving}\label{sec: sieving}
Lemma \ref{lem: SKcount} shows that  we simply need to compute $\#\Zstar(C\Oseen_K\D)\cap \IS_{C,\D}$.  
In this section we apply simple sieving arguments to reduce this task to an ordinary lattice point counting problem.

For the entire section we fix a non-zero ideal $\A$ in $\Oseen_K$ (playing the role of $C\Oseen_K\D$),
and an arbitrary finite subset  $\Sa$ of $\Oseen_K$  (playing the role of $\IS_{C,\D}$).
For a  non-zero ideal  $A\subseteq \Oseen_k$ we write
\begin{alignat*}1
\Z_A=\left(\A A\Oseen_K\backslash \bigcup_{\p|\Pc}\A\p A\Oseen_K\right)\cap \Sa.
\end{alignat*}
We note that 
\begin{alignat*}1
&\Z_A=\\
&\{\beta\in \A A\Oseen_K\backslash \{0\}; \beta\Oseen_K=\A A\Oseen_K\B \text{ for some $\B\subseteq \Oseen_K$ with }(\Pc,\B)=1\}
\cap \Sa.
\end{alignat*}

\begin{lemma}\label{lem: Zstarbasic}
We have 
\begin{alignat*}1
\Zstar(\A)\cap \Sa=\Z_{\Oseen_k}\backslash \bigcup_{A\subsetneq \Oseen_k}(\Z_{\Oseen_k}\cap \A A\Oseen_K)
\end{alignat*}
\end{lemma}
\begin{proof}
If $\beta\in \Zstar(\A)$ then clearly $\beta\in \Z_{\Oseen_k}$ and $\beta \notin \A A\Oseen_K$ whenever $A\subsetneq \Oseen_k$.
Hence, $\beta\in \Z_{\Oseen_k}\backslash \bigcup_{A\subsetneq \Oseen_k}(\Z_{\Oseen_k}\cap \A A\Oseen_K)$. 

Now suppose $\beta \in \Z_{\Oseen_k}\backslash \bigcup_{A\subsetneq \Oseen_k}(\Z_{\Oseen_k}\cap \A A\Oseen_K)$.
As  $\beta \in \Z_{\Oseen_k}$ we have  $\beta\Oseen_K=\A \Oseen_K\B$  for some $\B\subseteq \Oseen_K$ with $(\Pc,\B)=1$.
And since $\beta \notin \bigcup_{A\subsetneq \Oseen_k} \A A\Oseen_K$ we see that 
$A\Oseen_K\nmid \B \text{ for all }A\subsetneq \Oseen_k$. Hence, $\B\in I_\Pc$, and thus $\beta \in \Zstar(\A)$.
\end{proof}

\begin{lemma}\label{lem: ZQunion}
We have 
\begin{alignat*}1
\bigcup_{A\subsetneq \Oseen_k}(\Z_{\Oseen_k}\cap \A A\Oseen_K)=\bigcup_{Q \atop (Q,P_1\cdots P_s)=1}\Z_Q,
\end{alignat*}
where the union is taken over all prime ideals $Q\subseteq \Oseen_k$ different from the prime ideals $P_1,\ldots,P_s$. 
\end{lemma}
\begin{proof}
We have
\begin{alignat*}1
\bigcup_{A\subsetneq \Oseen_k}(\Z_{\Oseen_k}\cap \A A\Oseen_K)&=\bigcup_{Q}(\Z_{\Oseen_k}\cap \A Q\Oseen_K),
\end{alignat*}
and if $(Q,P_1\cdots P_s)\neq 1$ then 
\begin{alignat*}1
\Z_{\Oseen_k}\cap \A Q\Oseen_K=\emptyset.
\end{alignat*}
Therefore
\begin{alignat*}1
\bigcup_{A\subsetneq \Oseen_k}(\Z_{\Oseen_k}\cap \A A\Oseen_K)=
\bigcup_{Q \atop (Q,P_1\cdots P_s)=1}(\Z_{\Oseen_k}\cap \A Q\Oseen_K),
\end{alignat*}
and further,
\begin{alignat*}1
Z_{\Oseen_k}\cap \A Q\Oseen_K&=\Sa \cap\left(\A\backslash \bigcup_{\p|\Pc}\A\p\right)\cap\A Q\Oseen_K\\
&=\Sa \cap \A Q\Oseen_K\backslash \bigcup_{\p|\Pc}(\A\p\cap \A Q\Oseen_K).
\end{alignat*}
Finally, for prime $Q$ with $(Q,P_1\cdots P_s)=1$ we have $(Q\Oseen_K,\Pc)=1$ and thus 
$\A\p\Oseen_K\cap \A Q\Oseen_K=\A\p Q\Oseen_K$ whenever $\p|\Pc$. Hence,
\begin{alignat*}1
Z_{\Oseen_k}\cap \A Q\Oseen_K=\Sa \cap\A Q\Oseen_K\backslash \bigcup_{\p|\Pc}(\A\p Q\Oseen_K)=\Z_Q,
\end{alignat*}
and this completes the proof.
\end{proof}

\begin{lemma}\label{lem: ZQintersection}
If $Q_1,\ldots, Q_m$ are distinct prime ideals of $\Oseen_k$ all coprime to $P_1\cdots P_s$ then
\begin{alignat*}1
\bigcap_{i=1}^m \Z_{Q_i}=\Z_{Q_1\cdots Q_m}.
\end{alignat*}
\end{lemma}
\begin{proof}
For arbitrary  sets $A_i$ and subsets $B_{ij}$, and $A=\cap_i A_i$ one has 
$$\cap_i (A_i\backslash \cup_j B_{ij})=A\backslash(\cup_{i,j} B_{ij}\cap A).$$
Applying this with $A_i=\A Q_i\Oseen_K$, and $B_{ij}=\A \p_j Q_i\Oseen_K$, so that
$Z_{Q_i}=\Sa\cap A_i\backslash \cup_j B_{ij}$, and noting that $B_{ij}\cap A=\A \p_j Q_1\cdots Q_m$, as $(\p_j,Q_1\cdots Q_m\Oseen_K)=1$,
we get
\begin{alignat*}1
\bigcap_{i=1}^m \Z_{Q_i}=\Sa\cap \bigcap_{i=1}^m A_i\backslash \cup_j B_{ij}=\Sa\cap \A Q_1\cdots Q_m\backslash(\cup_{i,j}\A \p_j Q_1\cdots Q_m).
\end{alignat*}
Finally, we note that $\cup_{i,j}\A \p_j Q_1\cdots Q_m=\cup_{\p|\Pc}\A \p Q_1\cdots Q_m$, and thus the claim drops out.
\end{proof}

Let $\mu_k(\cdot)$ and $\mu_K(\cdot)$ be the M\"obius functions on non-zero  ideals of $\Oseen_k$ and  $\Oseen_K$ respectively.
\begin{lemma}\label{lem: countZQunion}
We have 
\begin{alignat*}1
\#\bigcup_{Q \atop (Q,P_1\cdots P_s)=1}\Z_{Q}=\sum_{A\subsetneq \Oseen_k \atop (A,P_1\cdots P_s)=1}-\mu_k(A)\#\Z_A,
\end{alignat*}
where the left union is taken over all prime ideals $Q\subseteq \Oseen_k$ different from the prime ideals $P_1,\ldots,P_s$. 
\end{lemma}
\begin{proof}
As the set $\Sa$ is finite there are only finitely many non-zero ideals $A\subsetneq \Oseen_k$   for which $\Z_{A}\neq \emptyset$.
Among those $A$ only finitely many prime ideals divisors $Q$ occur. 
Let $Q_1,\ldots,Q_m$ be those
that are coprime to $P_1\cdots P_s$ (if no such $Q$ exists then evidently both sides are $0$).
Applying the inclusion-exclusion principle, and then using Lemma \ref{lem: ZQintersection} we find
\begin{alignat*}1
\#\bigcup_{Q \atop (Q,P_1\cdots P_s)=1}\Z_{Q}=\#\bigcup_{i=1}^m\Z_{Q_i}&=\sum_{\emptyset\neq I\subseteq \{1,2,\ldots,m\}}(-1)^{\#I-1}\#\bigcap_{i\in I}\Z_{Q_i}\\
&=\sum_{\emptyset\neq I\subseteq \{1,2,\ldots,m\}}(-1)^{\#I-1}\#\Z_{\prod_{i\in I}Q_i}\\
&=\sum_{A\subsetneq \Oseen_k \atop (A,P_1\cdots P_s)=1}-\mu_k(A)\#\Z_A.
\end{alignat*}
\end{proof}

\begin{lemma}\label{lem: countZA}
Let $A$ be a non-zero ideal in $\Oseen_k$. Then we have 
\begin{alignat*}1
\#\Z_{A}=\sum_{\E|\Pc}\mu_K(\E)\#(\A\E A\Oseen_K\cap \Sa).
\end{alignat*}
\end{lemma}
\begin{proof}
The inclusion-exclusion principle yields
$$\#\bigcup_{\p|\Pc}\A\p A\Oseen_K\cap \Sa=\sum_{\E|\Pc \atop \E\neq \Oseen_K}-\mu_K(\E)\#(\A\E A\Oseen_K\cap \Sa).$$
As $\#\Z_{A}=\#(\A A\Oseen_K\cap \Sa)-\#\bigcup_{\p|\Pc}\A\p A\Oseen_K\cap \Sa$ the result follows at once.
\end{proof}

\begin{lemma}\label{lem: Zstarcount}
We have 
\begin{alignat*}1
\Zstar(\A)\cap \Sa=\sum_{\E|\Pc}\mu_K(\E)\left(\#(\A\E\cap \Sa)
+\sum_{A\subsetneq \Oseen_k \atop (A,P_1\cdots P_s)=1}\mu_k(A)\#(\A\E A\Oseen_K\cap \Sa)\right).
\end{alignat*}
\end{lemma}
\begin{proof}
Combining Lemma \ref{lem: Zstarbasic}, Lemma \ref{lem: ZQunion}, Lemma \ref{lem: countZQunion}, and Lemma \ref{lem: countZA}
the result drops out.
\end{proof}

\begin{lemma}\label{lem: constantspositive}
The following two identities hold
\begin{alignat*}1
\sum_{\E|\Pc}\frac{\mu_K(\E)}{N_{K/\IQ}(\E)}&=\prod_{P|\dis}\left(1-\frac{1}{N_{k/\IQ}(P)}\right),\\
\sum_{A\subseteq \Oseen_k \atop (A,P_1\cdots P_s)=1}\frac{\mu_k(A)}{N_{k/\IQ}(A)^2}&=\frac{1}{\zeta_k(2)}\prod_{P|\dis}\left(1-\frac{1}{N_{k/\IQ}(P)^2}\right)^{-1}.
\end{alignat*}
\end{lemma}
\begin{proof}
Recall that $\Pc=\p_1\cdots\p_s$.
If $s=0$  (i.e., $\Pc=\Oseen_K$) then both statements are obvious;  the first sum is $1$ and the second one is $\zeta_k(2)^{-1}$.
Note that $N_{k/\IQ}(P_i)=N_{K/\IQ}(\p_i)$. The first statement follows by induction from the following simple identity
\begin{alignat*}1
\sum_{\E|\p_1\cdots\p_{s}}\frac{\mu_K(\E)}{N_{K/\IQ}(\E)}=\sum_{\E|\p_1\cdots\p_{s-1}}\frac{\mu_K(\E)}{N_{K/\IQ}(\E)}\left(1+\frac{\mu_K(\p_s)}{N_{K/\IQ}(\p_s)}\right).
\end{alignat*}
Now let us prove the  second identity. We have
\begin{alignat*}1
\sum_{A\subseteq \Oseen_k \atop (A,P_1\cdots P_s)=1}\frac{\mu_k(A)}{N_{k/\IQ}(A)^2}&=\prod_{P\atop P\nmid P_1\cdots P_s}\left(\mu_k(1)+\mu_k(P)N_{k/\IQ}(P)^{-2}+\mu_k(P^2)N_{k/\IQ}(P)^{-4}+\cdots \right)\\
&=\prod_{P\atop P\nmid P_1\cdots P_s}\left(1-\frac{1}{N_{k/\IQ}(P)^2}\right)=\frac{1}{\zeta_k(2)}\prod_{P|\dis}\left(1-\frac{1}{N_{k/\IQ}(P)^2}\right)^{-1}.
\end{alignat*}
In the last step we have used that a prime ideal $P$ divides $P_1\cdots P_s$ if and only if it divides $\dis$.
\end{proof}

\section{Proof of Theorem \ref{thm: main}}\label{sec: proofmainthm}
Here we finalise the proof of Theorem \ref{thm: main}. 
Combining Lemma \ref{lem: SKcount} and Lemma \ref{lem: Zstarcount} we are led to the problem of counting elements of an ideal satisfying  certain archimedean conditions.  

Let $\F\subseteq \Oseen_K$ be a non-zero ideal; then $\bsigma(\F)$ is a lattice in $\IC^N\cong \IR^{2N}$ of determinant 
\begin{alignat}1\label{eq: det}
\det(\bsigma(\F))=2^{-N}N_{K/\IQ}(\F)\sqrt{|\Delta_K|},
\end{alignat}
and the shortest non-zero vector has euclidean length 
\begin{alignat}1\label{ineq: lambda}
\lambda_1(\bsigma(\F))\geq N_{K/\IQ}(\F)^{1/(2N)}
\end{alignat}
(see \cite[Lemma 5]{MasserVaaler2}).
For brevity let us write 
\begin{alignat*}3
V&=\Vol(S_F(\I^*;1)),\\
T&=\H N_{K/\IQ}(C\Oseen_K\D)^{1/(2N)}.
\end{alignat*}
Now 
\begin{alignat}1\label{eq: sigmaSCD}
\bsigma(\IS_{C,\D}))=\bsigma(C\Oseen_K\D)\cap S_F(\I^*;T),
\end{alignat}
and it follows from (\ref{eq: SFhom}) that 
\begin{alignat*}3
\Vol\left(S_F(\I^*;T)\right)=VT^{2N}=V\H^{2N}N_{K/\IQ}(C\Oseen_K\D).
\end{alignat*}

\begin{lemma}\label{lem: idealcount}
Let $C\in \R_k$, $\E|\Pc$, $\D|\Pc$,  $A \subseteq \Oseen_K$, and let $\H\geq 2$.
Then there exists a constant $c_1=c_1(K)$ such that
$$\left|\#(C\Oseen_K\D\E A\Oseen_K\cap \IS_{C,\D})-\frac{2^N V\H^{2N}}{\sqrt{|\Delta_K|}N_{K/\IQ}(\E)N_{k/\IQ}(A)^2}\right|\leq c_1\frac{\H^{2N-1}}{N_{k/\IQ}(A)^{2-1/N}}.$$
\end{lemma}
\begin{proof}
Using the injectivity of the map $\bsigma(\cdot)$ and (\ref{eq: sigmaSCD}) 
we get
\begin{alignat*}3
\#(C\Oseen_K\D\E A\Oseen_K\cap \IS_{C,\D})&=\#\left(\bsigma(C\Oseen_K\D\E A\Oseen_K)\cap \bsigma(\IS_{C,\D})\right)\\
&=\#\left(\bsigma(C\Oseen_K\D\E A\Oseen_K)\cap\bsigma(C\Oseen_K\D)\cap S_F(\I^*;T)\right)\\
&=\#\left(\bsigma(C\Oseen_K\D\E A\Oseen_K)\cap S_F(\I^*;T)\right).
\end{alignat*}

Combining  Lemma \ref{lem: Lip} and (\ref{eq: SFhom}) we see that $\partial(S_F(\I^*;T))$
is in Lip$(2N,M,L T)$, and that $S_F(\I^*;T)$ does not contain the origin but is contained in the zero centred ball of radius $LT$.
We apply Lemma \ref{lem: latticecounting} with $\Lambda=\bsigma(C\Oseen_K\D\E A\Oseen_K)$, and we use (\ref{eq: det}) and (\ref{ineq: lambda}). Since $$N_{K/\IQ}(\E A\Oseen_K)=N_{K/\IQ}(\E)N_{k/\IQ}(A)^2\geq N_{k/\IQ}(A)^2$$
 the result drops out.
\end{proof}

\begin{lemma}\label{lem: ZstarSCDcount}
Let $C\in \R_k$, $\D|\Pc$,  $\E|\Pc$, and $\H\geq 2$.
Then there exists a constant $c_2=c_2(K)$ such that
\begin{alignat*}1
\left|\sum_{A\subsetneq \Oseen_k \atop (A,P_1\cdots P_s)=1}\mu_k(A)\#(C\Oseen_K\D\E A\Oseen_K\cap \IS_{C,\D})-
\frac{2^N V\H^{2N}}{\sqrt{|\Delta_K|}N_{K/\IQ}(\E)}\sum_{A\subsetneq \Oseen_k \atop (A,P_1\cdots P_s)=1}\frac{\mu_k(A)}{N_{k/\IQ}(A)^2}\right| \\
\leq  c_2 \H^{2N-1}\IL,
\end{alignat*}
where $\IL=\log \H$ if $N=1$ and $\IL=1$ if $N\geq 2$.
\end{lemma}
\begin{proof}
For $N\geq 2$ this follows immediately from Lemma \ref{lem: idealcount}. If $N=1$ then we 
use that $C\Oseen_K\D\E A\Oseen_K\cap \IS_{C,\D}$ is empty whenever $N_{K/\IQ}(\E A\Oseen_K)>\H$. In particular, 
we can restrict the sum to those $A$ with $N_{k/\IQ}(A)\leq \H$. Applying Lemma \ref{lem: idealcount}
yields the error term 
$$\sum_{A\subsetneq \Oseen_k \atop {(A,P_1\cdots P_s)=1\atop N_{k/\IQ}(A)\leq \H}} c_1\frac{\H}{N_{k/\IQ}(A)}\ll_{K} \H\log \H.$$
Restricting the sum also  introduces the additional error term 
$$\left|\sum_{A\subsetneq \Oseen_k \atop {(A,P_1\cdots P_s)=1\atop N_{k/\IQ}(A)>\H}}
\frac{2^N V\H^{2}}{\sqrt{|\Delta_K|}N_{K/\IQ}(\E)}\frac{\mu_k(A)}{N_{k/\IQ}(A)^2}\right|\ll_{K}\H.$$
This completes the proof of Lemma \ref{lem: ZstarSCDcount}.
\end{proof}

Using 
Lemma \ref{lem: Zstarcount}, and then plugging in the estimates from Lemma \ref{lem: idealcount} and Lemma \ref{lem: ZstarSCDcount}
yields

\begin{alignat*}1
\#\Zstar(C\Oseen_K\D)\cap \IS_{C,\D}=&
\sum_{\E|\Pc}\frac{\mu_K(\E) 2^N V\H^{2N}}{\sqrt{|\Delta_K|}N_{K/\IQ}(\E)} \left(
1+\sum_{A\subsetneq \Oseen_k \atop (A,P_1\cdots P_s)=1}\frac{\mu_k(A)}{N_{k/\IQ}(A)^2}\right)\\
&+O_{K}\left(\H^{2N-1}\IL\right).
\end{alignat*}
For the main term we observe that 
\begin{alignat*}1
&\sum_{\E|\Pc}\frac{\mu_K(\E) 2^N V\H^{2N}}{\sqrt{|\Delta_K|}N_{K/\IQ}(\E)} \left(
1+\sum_{A\subsetneq \Oseen_k \atop (A,P_1\cdots P_s)=1}\frac{\mu_k(A)}{N_{k/\IQ}(A)^2}\right)\\
=&\frac{2^N V\H^{2N}}{\sqrt{|\Delta_K|}} 
\left(\sum_{\E|\Pc}\frac{\mu_K(\E)}{N_{K/\IQ}(\E)}\right)
\left(\sum_{A\subseteq \Oseen_k \atop (A,P_1\cdots P_s)=1}\frac{\mu_k(A)}{N_{k/\IQ}(A)^2}\right).
\end{alignat*}

Now using Lemma \ref{lem: SKcount} and summing over the $2^s$ divisors  $\D|\Pc$, and then over $C\in \R_k$ we get

\begin{alignat*}1
\Sc_K(\H; \I)=&\frac{h_k 2^{s+N} V\H^{2N}}{\sqrt{|\Delta_K|}} 
\left(\sum_{\E|\Pc}\frac{\mu_K(\E)}{N_{K/\IQ}(\E)}\right)
\left(\sum_{A\subseteq \Oseen_k \atop (A,P_1\cdots P_s)=1}\frac{\mu_k(A)}{N_{k/\IQ}(A)^2}\right)\\
&+O_{K}\left(2^sh_k\H^{2N-1}\IL\right).
\end{alignat*}

Using Lemma \ref{lem: constantspositive} and recalling that $\dis$ has exactly $s$ prime ideal
factors gives
\begin{alignat*}1
2^{s} \left(\sum_{\E|\Pc}\frac{\mu_K(\E)}{N_{K/\IQ}(\E)}\right)
\left(\sum_{A\subseteq \Oseen_k \atop (A,P_1\cdots P_s)=1}\frac{\mu_k(A)}{N_{k/\IQ}(A)^2}\right)
=\left(\prod_{P|\dis}\frac{2N_{k/\IQ}(P)}{N_{k/\IQ}(P)+1}\right)\frac{1}{\zeta_k(2)}.
\end{alignat*}
Plugging in the value for $V$ from Lemma \ref{lem: volume}, and using 
$$|\Delta_K|=|\Delta_k|^2N_{k/\IQ}(\dis)$$ 
(see \cite[p.24]{13}) shows that the leading constant of the main term is given by $\conSK|\I|$.
This completes the proof of Theorem \ref{thm: main}.

\section{Minimal heights in cosets of $k^\times$}
In this section we generalise the previous CM-field setting. We assume throughout that $K/k$
is an arbitrary quadratic extension of number fields.

Let $\tau: K\to K$ be the unique automorphism
that fixes $k$, so that $\Gal(K/k)=\langle \tau\rangle$, and $\Norm(\alpha)=\alpha\tau(\alpha)$.
As in Section \ref{sec:basics} we let
$\psi:K^\times\to K^\times$ be the group homomorphism defined by  
\begin{alignat*}1
\psi(\beta)=\frac{\beta}{\tau(\beta)},
\end{alignat*}
so that $\ker \psi=k^\times$. We write $\NKk$ for the kernel of the norm map $\Norm$.
Hence,  Hilbert's Theorem 90 implies $\IM \psi=\NKk$, and we get an induced isomorphism
\begin{alignat}1\label{map:isom}
\hat{\psi}: K^\times/k^\times \to \NKk.
\end{alignat}
We will determine elements of minimal height for those cosets of $k^\times$ in $K^\times$
that intersect $\NKk$.

\begin{lemma}\label{lemfirst5}   If $\alpha$ belongs to $\NKk$ then the inequality
\begin{equation}\label{eleventh238}
H(\alpha) \le H(\alpha \gamma)
\end{equation}
holds for each $\gamma$ in $k^{\times}$.
In particular, the minimum value of the Weil height on 
elements of the multiplicative coset $\alpha k^{\times}$ is given by
\begin{equation*}\label{eleventh280}
\min\big\{H(\alpha \gamma) : \gamma \in k^{\times}\big\} = H(\alpha).
\end{equation*}     
\end{lemma}

\begin{proof}  Assume that $\alpha$ is in $\NKk$ and $\gamma$ is in $k^{\times}$.  The automorphism $\tau$ preserves the height of points in 
$K^{\times}$.  Therefore $\alpha \gamma$ and
\begin{equation*}
\tau(\alpha \gamma) = \tau(\alpha) \tau(\gamma) = \tau(\alpha) \gamma = \alpha^{-1} \gamma
\end{equation*}
have the same height.  Similarly, $\alpha^{-1} \gamma$ and
\begin{equation*}\label{eleventh252}
\bigl(\alpha^{-1} \gamma\bigr)^{-1} = \alpha \gamma^{-1}
\end{equation*}
have the same height.  That is, the three elements
\begin{equation*}\label{eleventh259}
\alpha \gamma, \quad \alpha^{-1} \gamma,\quad\text{and}\quad \alpha \gamma^{-1},
\end{equation*}
satisfy the identity
\begin{equation*}
H(\alpha \gamma) = H\bigl(\alpha^{-1} \gamma\bigr) = H\bigl(\alpha \gamma^{-1}\bigr).
\end{equation*}
Now by well known properties of the height we get
\begin{equation*}
\begin{split}
H(\alpha)^2 &= H\bigl(\alpha^2\bigr) = H\bigl(\bigl(\alpha \gamma\bigr) \bigl(\alpha \gamma^{-1}\bigr)\bigr)\\
		&\le H(\alpha \gamma)H\bigl(\alpha \gamma^{-1}\bigr) = H(\alpha \gamma)^2.
\end{split}
\end{equation*}  
This verifies the inequality  (\ref{eleventh238}).
\end{proof}

Consider the inverse of the isomorphism (\ref{map:isom})
$${\hat{\psi}}^{-1}:\NKk \to K^\times/k^\times.$$
Lemma \ref{lemfirst5} raises the following question. Which elements of $\NKk$ 
are mapped under ${\hat{\psi}}^{-1}$ to cosets that intersect $\NKk$?
This question is answered by  Lemma \ref{lem:SKcosets} which follows easily from the 
following simple observation.

\begin{lemma}\label{lem:psionSK}
An element $\alpha \in K^\times$ lies in $\NKk$ if and only if 
$\psi(\alpha)=\alpha^2.$
\end{lemma}
\begin{proof}
Since $\psi^2(\beta)=\psi(\beta)^2$ for any $\beta\in K^\times$ we get $\psi(\alpha)=\alpha^2$ whenever $\alpha \in \NKk$.
And if $\psi(\alpha)=\alpha^2$ then $\alpha\tau(\alpha)=1$, and thus $\alpha \in \NKk$.
This proves the lemma.
\end{proof}

\begin{lemma}\label{lem:SKcosets}
A coset of $k^\times$ in $K^\times$ intersects $\NKk$ if and only if it is the image 
of a square in $\NKk$ under the isomorphism ${\hat{\psi}}^{-1}$.
\end{lemma}
\begin{proof}

First suppose $\beta\in {\NKk}$. Then $\psi(\beta)=\beta^2$ by Lemma \ref{lem:psionSK},  and we get
$${\hat{\psi}}^{-1}(\beta^2)={\hat{\psi}}^{-1}(\psi(\beta))=\beta k^\times,$$
proving that the image of a square is a coset that intersects $\NKk$.

Next suppose that the image $\beta k^\times$ intersects $\NKk$. We can assume 
$\beta\in \NKk$, and thus  $\psi(\beta)=\beta^2$ by Lemma \ref{lem:psionSK}. 
We conclude   
$$\beta k^\times={\hat{\psi}}^{-1}(\psi(\beta))={\hat{\psi}}^{-1}(\beta^2),$$
which proves the other direction.
\end{proof}

\bibliographystyle{amsplain}
\bibliography{literature, Books}

\end{document}